\documentclass{amsart}
\usepackage{amssymb,latexsym,amsmath,amscd,graphicx,setspace,amsthm,verbatim,comment,stmaryrd,url}
\usepackage[french,english]{babel}
\usepackage{tikz,tikz-cd}
\usepackage[margin = 3 cm]{geometry} \input xy \xyoption{all}
\usepackage{hyperref}

\hypersetup{
 colorlinks=true
 }

\usetikzlibrary{shapes.geometric}

\tikzset{every loop/.style={min distance=10mm,looseness=10}}

\newcount\mycount

\theoremstyle{plain}
\newtheorem{theorem}{Theorem}[section]
\newtheorem{proposition}[theorem]{Proposition}

\newtheorem{corollary}[theorem]{Corollary}

\newtheorem{lthm}{Theorem} % theorems with letters (for intro)

\makeatletter
\def\th@remark{%
  \thm@headfont{\bfseries}%
  \normalfont % body font
  \thm@preskip\topsep \divide\thm@preskip\tw@
  \thm@postskip\thm@preskip
}
\makeatother

\theoremstyle{remark} 
\newtheorem{remark}[theorem]{Remark}
      
\newtheorem{Ex}[theorem]{Example}
\theoremstyle{definition}
\newtheorem{definition}[theorem]{Definition}

\definecolor{Green}{rgb}{0.0, 0.5, 0.0}
\definecolor{Blue}{rgb}{0.0,0.0,1.0}

\newcommand{\bz}{\mathbf{z}}
\newcommand{\QQ}{\mathbb{Q}}
\newcommand{\ZZ}{\mathbb{Z}}
\newcommand{\Zl}{\ZZ_\ell}

\newcommand{\ord}{\mathrm{ord}}

\begin{document} 
\title[The non-$\ell$-part of the number of spanning trees in abelian $\ell$-towers of multigraphs]{The non-$\ell$-part of the number of spanning trees in abelian $\ell$-towers of multigraphs}

\author{Antonio Lei, Daniel Valli\`{e}res} 

%FILL IN

\address{Antonio Lei\newline Department of Mathematics and Statistics\\University of Ottawa\\
150 Louis-Pasteur Pvt\\
Ottawa, ON\\
Canada K1N 6N5}
\email{antonio.lei@uottawa.ca}

\address{Daniel Valli\`{e}res\newline Mathematics and Statistics Department, California State University, Chico, CA 95929 USA}
\email{dvallieres@csuchico.edu}

\date{\today} 
\begin{abstract}
Let $\ell$ and $p$ be two distinct primes. We study the $p$-adic valuation of the number of spanning trees in an abelian $\ell$-tower of connected multigraphs. This is analogous to the classical theorem of Washington--Sinnott on the growth of the $p$-part of the class group in a cyclotomic $\Zl$-extension of abelian extensions of $\QQ$. Furthermore, we show that under certain hypotheses, the number of primes dividing the number of spanning trees is unbounded in such a tower.
\end{abstract} 
\maketitle 
\tableofcontents 
\addtocontents{toc}{\protect\setcounter{tocdepth}{1}}

%\section{Introduction}
\section{Introduction}
Let $\ell$ be a rational prime and let $K$ be a number field.  Consider a $\mathbb{Z}_{\ell}$-extension of $K$ and its associated tower of number fields
$$K = K_{0} \subseteq K_{1} \subseteq K_{2} \subseteq \ldots \subseteq K_{n} \subseteq \ldots  $$
meaning that $K_n/K$ is a Galois extension with Galois group isomorphic to $\ZZ/\ell^n\ZZ$. A classical theorem of Iwasawa (see \cite{Iwasawa:1959} and \cite{Iwasawa:1973}) predicts the existence of non-negative integers $\mu_{\ell}$, $\lambda_{\ell}$ and an integer $\nu_{\ell}$ such that
$${\rm ord}_{\ell}(h_{n}) = \mu_{\ell} \cdot \ell^{n} + \lambda_{\ell} \cdot n + \nu_{\ell} $$
for $n$ large enough, where $h_{n}$ is the class number of $K_{n}$, and ${\rm ord}_{\ell}$ denotes the usual $\ell$-adic valuation on $\mathbb{Q}$.  Let $p$ be another rational prime different from $\ell$.  In \cite{Washington:1975}, Washington initiated the study of the $p$-adic valuation of $h_{n}$ in  towers as above, and in \cite{Washington:1978}, he proved that in the situation where $K$ is abelian over $\mathbb{Q}$, and the $\mathbb{Z}_{\ell}$-extension is the cyclotomic one, there exists an integer $\nu_{p}$ such that one has
$${\rm ord}_{p}(h_{n}) = \nu_{p} $$
for $n$ large enough.  Sinnott gave a different proof of Washington's result in \cite{Sinnott:1987}.

In \cite{Vallieres:2021}, the second-named author introduced the notion of an abelian $\ell$-tower of connected multigraphs, which can be viewed as  analogues of $\mathbb{Z}_{\ell}$-extensions of number fields.  If $X$ is a connected multigraph, such a tower consists of a sequence of covers of connected multigraphs
\begin{equation} \label{gen_tower}
X = X_{0} \longleftarrow X_{1} \longleftarrow X_{2} \longleftarrow \ldots \longleftarrow X_{n} \longleftarrow \ldots
\end{equation}
for which the cover $X_{n}/X$ obtained by composing the covers $X_{n} \longrightarrow X_{n-1} \longrightarrow \ldots \longrightarrow X$ is Galois with group of covering transformations isomorphic to $\mathbb{Z}/\ell^{n}\mathbb{Z}$.  In \cite{Vallieres:2021}, \cite{McGown/Vallieres:2022}, and \cite{McGown/Vallieres:2022a}, it was shown that the $\ell$-adic valuation of the number of spanning trees of $X_{n}$ behaves just as in Iwasawa's theorem.  More precisely, there exist non-negative integers $\mu_{\ell}$, $\lambda_{\ell}$ and an integer $\nu_{\ell}$ such that
$${\rm ord}_{\ell}(\kappa_{n}) = \mu_{\ell}\cdot \ell^{n} + \lambda_{\ell}\cdot n + \nu_{\ell} $$
for $n$ large enough, where $\kappa_{n}$ is the number of spanning trees of $X_{n}$.  See also \cite{Gonet:2021} and \cite{Gonet:2022} for similar results proved independently using different methods.  The goal of the current paper is to study the $p$-adic valuation of $\kappa_{n}$ in an abelian $\ell$-tower of connected multigraphs for $p\neq \ell$, by analogy with the Washington-Sinnott theorem on cyclotomic $\Zl$-extensions of abelian number fields. In particular, we prove the following:

\begin{lthm}[{Theorem~\ref{main_1}}]\label{thmA}
Let $\ell$ and $p$ be two distinct prime numbers.  Let $X$ be a connected multigraph and $$X = X_{0} \longleftarrow X_{1} \longleftarrow X_{2} \longleftarrow \ldots \longleftarrow X_{n} \longleftarrow \ldots$$
 an abelian $\ell$-tower of connected multigraphs. There exist non-negative integers $n_0$, $\mu_p$ and an integer $\nu_p$ such that
 \[
 \ord_p(\kappa_n)=\mu_p\cdot \ell^n+\nu_p,
 \]
 for all $n\ge n_0$, where $\kappa_n$ denotes the number of spanning trees of $X_n$.
\end{lthm}

The main idea of the proof of Theorem~\ref{thmA} is to relate $\kappa_n$ to the value of a certain generalized polynomial  $f(T)$ at $\ell$-power roots of unity. This allows us to make use of a result of Sinnott from \cite{Sinnott:1987} to study the $p$-adic valuations of these values.  Under an additional hypothesis on the aforementioned generalized polynomial $f(T)$, we employ a result of Schinzel \cite{Schinzel:1974} to give a sufficient and necessary condition for the number of primes dividing $\kappa_n$ to be unbounded. 

\begin{lthm}[{Theorem~\ref{number_of_primes}}]\label{thmB}
Under the same notation as Theorem~\ref{thmA}, suppose that all exponents of $f(T)$ are in $\ZZ$. The number of primes dividing $\kappa_n$ is unbounded as $n\rightarrow\infty$  if and only if $f(T)$ admits a root $\alpha$ in $\overline{\mathbb{Q}}$ which is not a root of unity.
\end{lthm}

A similar result in the setting of ideal class groups has been proved by Washington. See \cite[Corollary~3]{Washington:1975}, which says that if $K$ is an  abelian CM field, then the number of primes dividing $h_n^-$ tends to infinity as $n\rightarrow\infty$, where $h_n^-$ is the relative class number of the $n$-th layer of the cyclotomic $\Zl$-extension of $K$.

\begin{remark}
Let $L$ be an $r$-component link in the 3-sphere $S^3$ and write $X$ for the complement of an open tubular neighbourhood of $L$. Let $G_{L}$ be the link group of $L$ and let $[m_{i}] \in G_{L}$ be the class of the meridian $m_{i}$.  We have an isomorphism $H_{1}(X,\mathbb{Z}) \simeq G_{L}/G_{L}'$ and we let $t_{i} \in H_{1}(X,\mathbb{Z})$ correspond to $[m_{i}]$ via this previous isomorphism.  Given an $r$-tuple  $\bz=(z_1,\ldots,z_r)\in\ZZ^r$ such that gcd$(z_1,\ldots,z_r)=1$, there exists a $\Zl$-cover $X_\bz=(X_{\bz,\ell^n})_{n\ge0}$, meaning that $X_{\bz,\ell^n}\rightarrow X$ is an $\ell^n$-fold cyclic cover for all $n$, constructed from the surjective morphism $H_{1}(X,\mathbb{Z}) \longrightarrow \mathbb{Z}$ given by $t_{i} \mapsto z_{i}$. Let $M_{\bz,\ell^n}\rightarrow S^3$ be the Fox completion of $X_{\bz,\ell^n}\rightarrow X$. Then, \cite[Theorem~2.1(1)]{KM-link} tells us that for all $n\ge v:=\max\{\ord_{\ell}(z_i):1\le i\le r\}$, we have
\begin{equation}
\left|H_1(M_{\bz,\ell^n};\ZZ)\right|=\left|H_1(M_{\bz,\ell^v};\ZZ)\right|\cdot\left|\prod_{\substack{\zeta^{\ell^n}=1\\\zeta^{\ell^v}\neq1}}\Delta_{L,\bz}(\zeta)\right|,
\label{eq:formula-link}
\end{equation}
where $\Delta_{L,\bz}(t)$ is the reduced Alexander polynomial attached to $X_\bz$. It is of the form $(t-1)\Delta_L(t^{z_1},\ldots,t^{z_r})$  if $r\ge 2$ and of the form $\Delta_{L}(t)$ if $r =1$, where $\Delta_L$ is the Alexander polynomial of $L$. An asymptotic formula for $\ord_\ell\left(\left|H_1(M_{\bz,\ell^n};\ZZ)\right|\right)$ has been obtained by Hillman--Matei--Morishita (see \cite[Theorem~5.1.7]{HMM-link} or \cite[Theorem~2.1(2)]{KM-link}). Since the formula \eqref{eq:formula-link} is essentially of the same form as the one for $\kappa_n$ used in the proof of Theorem~\ref{thmA}, we expect that our method can be readily generalized  to deduce an asymptotic formula for $\ord_p\left(\left|H_1(M_{\bz,\ell^n};\ZZ)\right|\right)$ for $n\gg0$, where $p$ is a prime number not equal to $\ell$.

In a recent preprint \cite{CR-link}, Dion and Ray studied the statistical behaviour of the Iwasawa invariants  of certain covers of the $3$-sphere that are branched along a link. It seems reasonable to expect that some of their techniques can be generalized to study similar questions in the context of Iwasawa theory of coverings of graphs.  
\end{remark}

The paper is organized as follows.  After presenting a few preliminaries in \S\ref{Sec:main_0}, we prove Theorem~\ref{thmA} in \S\ref{Sec:main_1} and deduce some immediate consequences on the $\kappa_n$. We then prove Theorem~\ref{thmB} in \S\ref{Sec:main_2}. Finally, we illustrate our results with a few explicit examples in \S\ref{Sec:Ex}.

\subsection*{Acknowledgement}
The authors would like to thank Cédric Dion and Anwesh Ray for interesting discussions on topics related to the present article. DV would like to thank the pure mathematics group at the California State University - Chico including John Lind, Thomas Mattman, and Kevin McGown for several stimulating discussions during our seminar.  DV would also like to thank AL for inviting him to give a talk at the Universit\'{e} Laval on related topics from which the current paper is a result of discussions that followed.  AL's research is supported by the NSERC Discovery Grants Program RGPIN-2020-04259 and RGPAS-2020-00096. AL would like to thank DV for introducing this beautiful subject to him.

%\subsection{Preliminaries}
\subsection{Preliminaries} \label{Sec:main_0}
We use the same notation and terminology for multigraphs as in \cite[\S$4$]{McGown/Vallieres:2022a}.  (Our main reference for multigraphs, covers of multigraphs, and related notions is \cite{Sunada:2013}.  For Artin-Ihara $L$-functions, our main reference is \cite{Terras:2011}.)  \emph{In particular, all multigraphs arising in this paper will be connected, finite and with no vertex of valency one.  Furthermore, we shall assume throughout that the Euler characteristic of our multigraphs is not zero.}  (This case is simple and can be treated separately.  See the discussion after \cite[Definition $4.1$]{Vallieres:2021}.)  If $Y/X$ is an abelian cover of connected multigraphs with group of covering transformations $G$ and $\psi$ is a character of $G$, then the corresponding Artin-Ihara $L$-function will be denoted by $L_{X}(u,\psi)$.  The three-term determinant formula for Artin-Ihara $L$-functions implies that
$$\frac{1}{L_{X}(u,\psi)} = (1-u^{2})^{-\chi(X)}\cdot h_{X}(u,\psi), $$
where $\chi(X)$ is the Euler characteristic of $X$ and
\begin{equation} \label{main_l}
h_X(u,\psi)=\det(I-A_\psi u +(D-I)u^2)\in\ZZ[\psi][u].
\end{equation}
Here, $A_\psi$ is the adjacency matrix of $X$ twisted by $\psi$, $D$ the valency matrix of $X$ and $\mathbb{Z}[\psi]$ is the ring of integers in the cyclotomic number field $\mathbb{Q}(\psi)$.  For more details, see \cite[\S$2.2$]{Vallieres:2021}.  Our starting point is   \cite[Equation $(7)$]{Vallieres:2021}:  if $Y/X$ is an abelian cover of multigraphs with Galois group of covering transformations $G$, then
\begin{equation} \label{eq}
|G| \cdot \kappa_{Y} = \kappa_{X} \prod_{\Psi \neq \Psi_{0}} h_{X}(1,\Psi),
\end{equation}
where $\kappa$ denotes the number of spanning trees of a multigraph, and the product is over all the non-trivial orbits of the ${\rm Gal}(\overline{\mathbb{Q}}/\mathbb{Q})$-set $\widehat{G} = {\rm Hom}_{\mathbb{Z}}(G,\mathbb{C}^{\times})$.  In this last equation and the forthcoming ones, if $\Omega$ is a finite collection of characters, we always set
$$h_{X}(s,\Omega)  = \prod_{\psi \in \Omega} h_{X}(s,\psi).$$  

Throughout this paper, we fix a rational prime $\ell$.  For $m \in \mathbb{Z}_{\ge 1}$, we set
$$\zeta_{m} = {\rm exp}\left(\frac{2 \pi i}{m}\right) \in \overline{\mathbb{Q}} \subseteq \mathbb{C},$$
and we let
$$\QQ(\zeta_{\ell^\infty}) = \mathbb{Q}(\{\zeta_{\ell^{i}} \, | \, i=1,2,\ldots \}).$$
From now on, let $p$ be a fixed rational prime different from $\ell$ and fix an embedding $\tau:\mathbb{Q}(\zeta_{\ell^{\infty}}) \hookrightarrow \overline{\mathbb{Q}}_{p}$.  We also let $\xi_{\ell^{i}} = \tau(\zeta_{\ell^{i}}) \in \overline{\mathbb{Q}}_{p}$ and we denote by $v_{p}$ the valuation on $\overline{\mathbb{Q}}_{p}$ that extends the discrete valuation ${\rm ord}_{p}$ on $\mathbb{Q}_{p}$.  For each $i \ge 1$, we let $\mathfrak{p}_{i}$ be the prime ideal in $\mathbb{Q}(\zeta_{\ell^{i}})$ lying above $p$ which is determined by the embedding $\tau$.  If $x \in \mathbb{Q}(\zeta_{\ell^{i}})$, then we have
\begin{equation} \label{val_relation}
v_{p}(\tau(x)) = {\rm ord}_{\mathfrak{p}_{i}}(x).
\end{equation}
We have an infinite tower of unramified extensions of local fields
$$\mathbb{Q}_{p} \subseteq \mathbb{Q}_{p}(\xi_{\ell}) \subseteq \mathbb{Q}_{p}(\xi_{\ell^{2}}) \subseteq \ldots \subseteq \mathbb{Q}_{p}(\xi_{\ell^{n}}) \subseteq \ldots$$
that are all included in 
$$\mathbb{Q}_{p}(\xi_{\ell^{\infty}}) := \bigcup_{i \ge 0} \mathbb{Q}_{p}(\xi_{\ell^{i}}).$$
The valuation $v_{p}$ on $\overline{\mathbb{Q}}_{p}$ restricted to $\mathbb{Q}_{p}(\xi_{\ell^{i}})$ is discrete for $i=0,1,2,\ldots,\infty$.  We let $O_{i}$ be the corresponding valuation ring and $\mathfrak{m}_{i}$ the unique maximal ideal of $O_{i}$.  We also let $k_{i} = O_{i}/\mathfrak{m}_{i}$ be the corresponding residue field.  We obtain a tower of finite fields
$$\mathbb{F}_{p} \subseteq k_{1} \subseteq k_{2} \subseteq \ldots \subseteq k_{n} \subseteq \ldots $$
that are all contained in $k_{\infty}$.

If $R$ is any unital commutative ring, then we let 
$$\mu_{\ell^{\infty}}(R) = \{\lambda \in R^{\times} \, | \, \lambda^{\ell^{n}} = 1 \text{ for some } n=0,1,2,\ldots\}.$$
We will now consider various elements in the group ring $R[\mathbb{Z}_{\ell}]$, but it will be convenient to think of elements in $R[\mathbb{Z}_{\ell}]$ as generalized polynomials as is done for instance in \cite{Gilmer:1984}.  Thus we introduce a formal variable $T$ and we consider $R[T;\mathbb{Z}_{\ell}]$ whose elements are polynomials in $T$ with exponents in $\mathbb{Z}_{\ell}$ and coefficients in $R$.  Elements of $R[T;\mathbb{Z}_{\ell}]$ are expressions of the form
$$P(T) = \sum_{a \in \mathbb{Z}_{\ell}}\lambda_{a}T^{a}, $$ 
where $\lambda_{a} \in R$ is zero for all but finitely many $a \in \mathbb{Z}_{\ell}$.  Generalized polynomials are added and multiplied just as regular polynomials are.  Note that $R[T] = R[T;\mathbb{Z}_{\ge 0}]$, $R[T,T^{-1}] = R[T;\mathbb{Z}]$, and $R[\mathbb{Z}_{\ell}]\simeq R[T;\mathbb{Z}_{\ell}]$ via $a \in \mathbb{Z}_{\ell}\mapsto T^{a} \in \mathbb{Z}[T;\mathbb{Z}_{\ell}]$.  Let now $S$ be another unital commutative ring and assume we are given a unital ring morphism $R \longrightarrow S$.  If $\xi \in \mu_{\ell^{\infty}}(S)$, then we have an $R$-algebra morphism ${\rm ev}_{\xi}:R[T;\mathbb{Z}_{\ell}] \longrightarrow S$ defined via
$$P(T) \mapsto {\rm ev}_{\xi}(P(T)) = P(\xi) := \sum_{a \in \mathbb{Z}_{\ell}}\lambda_{a}\xi^{a}. $$
Any generalized polynomial $P(T) \in R[T;\mathbb{Z}_{\ell}]$ induces a function $P:\mu_{\ell^{\infty}}(S) \longrightarrow S$ via 
\begin{equation*} 
\xi \mapsto {\rm ev}_{\xi}(P(T)) = P(\xi),
\end{equation*}
and in particular, one can talk about the zeros of a generalized polynomial in $\mu_{\ell^{\infty}}(S)$.  
\begin{definition}
Let $P(T) = \sum_{a \in \mathbb{Z}_{\ell}}\lambda_{a}T^{a} \in \mathbb{Z}[T;\mathbb{Z}_{\ell}]$, then we define
$$\mu(P(T)) =  {\rm min}\{{\rm ord}_{p}(\lambda_{a}) \, | \, a \in \mathbb{Z}_{\ell} \}.$$
\end{definition}
There is a projection map $\mathbb{Z}[T;\mathbb{Z}_{\ell}] \longrightarrow \mathbb{F}_{p}[T;\mathbb{Z}_{\ell}]$ defined via
$$P(T) = \sum_{a \in \mathbb{Z}_{\ell}}\lambda_{a}T^{a} \mapsto \bar{P}(T) = \sum_{a \in \mathbb{Z}_{\ell}}\bar{\lambda}_{a}T^{a},$$
where the bar denotes reduction modulo $p$.
\begin{theorem}[Sinnott] \label{Sinnott}
Let $P(T) \in \mathbb{Z}[T;\mathbb{Z}_{\ell}]$ be such that $\mu(P(T)) = 0$, then $\bar{P}(T) \in \mathbb{F}_{p}[T;\mathbb{Z}_{\ell}]$ has only finitely many zeros in $\mu_{\ell^{\infty}}(k_{\infty})$.
\end{theorem}
\begin{proof}
This follows directly from \cite[Theorem~2.2]{Sinnott:1987}.
\end{proof}

%\section{Abelian $\ell$-towers of multigraphs}
\section{Abelian $\ell$-towers of multigraphs}\label{Sec:main_1}
Recall from \cite[\S $4$]{McGown/Vallieres:2022a} that for any connected multigraph $X$, we have an incidence function
$${\rm inc}:\vec{E}_{X}\longrightarrow V_{X} \times V_{X},$$
where $\vec{E}_{X}$ denotes the directed edges and $V_{X}$ the vertices of $X$.  Let $\gamma:E_{X} \longrightarrow \vec{E}_{X}$ be a section of the natural map $\vec{E}_{X}\longrightarrow E_{X}$, where $E_{X}$ denotes the undirected edges.  Set also $S=\gamma(E_{X})$.   Further, let $\alpha:S \longrightarrow \mathbb{Z}_{\ell}$ be a function, and let $\alpha_{n}$ denote the function $S \longrightarrow \mathbb{Z}/\ell^{n}\mathbb{Z}$ given by the composition of $\alpha$ and the natural projection $\mathbb{Z}_{\ell} \longrightarrow \mathbb{Z}/\ell^{n}\mathbb{Z}$.  We assume that all the derived multigraphs $X(\mathbb{Z}/\ell^{n}\mathbb{Z},S,\alpha_{n})$ are connected.  In this case, we obtain an abelian $\ell$-tower of connected multigraphs
\begin{equation} \label{concrete_tower}
X  \longleftarrow X(\mathbb{Z}/\ell\mathbb{Z},S,\alpha_{1}) \longleftarrow X(\mathbb{Z}/\ell^{2}\mathbb{Z},S,\alpha_{2}) \longleftarrow \ldots \longleftarrow X(\mathbb{Z}/\ell^{n}\mathbb{Z},S,\alpha_{n}) \longleftarrow \ldots,
\end{equation}
as explained in \cite[\S$4$]{McGown/Vallieres:2022a}.  The group of covering transformations of the cover $X(\mathbb{Z}/\ell^{n}\mathbb{Z},S,\alpha_{n}) \longrightarrow X$ is isomorphic to $\mathbb{Z}/\ell^{n}\mathbb{Z}$, which we shall denote by $\Gamma_{n}$ from now on.  Label the vertices $V_{X} = \{v_{1},\ldots,v_{g} \}$, and consider the matrix
$$M(T) =  D - \left(\sum_{\substack{s \in S \\ {\rm inc}(s) = (v_{i},v_{j})}} T^{\alpha(s)} +  \sum_{\substack{s \in S \\ {\rm inc}(s) = (v_{j},v_{i})}} T^{-\alpha(s)} \right) \in M_{g \times g}(\mathbb{Z}[T;\mathbb{Z}_{\ell}]),$$
where $D$ is the valency matrix of $X$.

We now prove Theorem~\ref{thmA} stated in the introduction:
\begin{theorem} \label{main_1}
With the notation as above, define
$$f(T) = {\rm det}(M(T)) \in \mathbb{Z}[T;\mathbb{Z}_{\ell}]. $$
Let $\mu = \mu(f(T))$ and write
$$f(T) = p^{\mu}\cdot g(T),$$
for some $g(T) \in \mathbb{Z}[T;\mathbb{Z}_{\ell}]$ satisfying $\mu(g(T)) = 0$.  Let $n_{0}$ be the minimal positive integer (which exists by Theorem \ref{Sinnott}) such that
$$g(\xi) \not \equiv 0 \pmod{\mathfrak{m}_{\infty}}, $$
whenever $\xi$ is a primitive $\ell^{i}$-th root of unity with $i \ge n_{0}$.  Then, there exists an integer $\nu$ such that
$${\rm ord}_{p}(\kappa_{n}) = \mu \cdot \ell^{n} + \nu,$$
when $n \ge n_{0}$.
\end{theorem}
\begin{proof}
As explained in \cite[\S$5.1$]{Vallieres:2021}, the inflation property of the Artin-Ihara $L$-functions combined with (\ref{eq}) imply that for $n \ge 1$, we have
\begin{equation} \label{starting_point}
\ell^{n} \cdot \kappa_{n} = \kappa_{X} \cdot \prod_{i=1}^{n} h_{X}(1,\Psi_{i}), 
\end{equation}
where $\Psi_{i}$ consists of the faithful characters of $\Gamma_{i}$ and $\kappa_{n}$ is the number of spanning tree of $X(\mathbb{Z}/\ell^{n}\mathbb{Z},S,\alpha_{n})$.  Applying $\tau$ first, then $v_{p}$ to (\ref{starting_point}), and using (\ref{val_relation}), we obtain
\begin{equation*} 
{\rm ord}_{p}(\kappa_{n}) = C + \sum_{i=n_{0}}^{n} v_{p}(\tau(h_{X}(1,\Psi_{i}))) =  C + \sum_{i=n_{0}}^{n} {\rm ord}_{\mathfrak{p}_{i}}(h_{X}(1,\Psi_{i})), 
\end{equation*}
for some constant $C$ whenever $n \ge n_{0}$.   We let $\psi_{i}$ be the character of $\Gamma_{i}$ satisfying $\psi_{i}(\bar{1}) = \zeta_{\ell^{i}}$.  We then have
\begin{equation*}
\begin{aligned}
{\rm ord}_{\mathfrak{p}_{i}}(h_{X}(1,\Psi_{i})) &= \sum_{\psi \in \Psi_{i}}{\rm ord}_{\mathfrak{p}_{i}}(h_{X}(1,\psi)) \\
&= \sum_{\sigma \in (\mathbb{Z}/\ell^{i}\mathbb{Z})^{\times}}{\rm ord}_{\mathfrak{p}_{i}}(h_{X}(1,\psi_{i}^{\sigma})), \\
\end{aligned}
\end{equation*}
and we are left to understand ${\rm ord}_{\mathfrak{p}_{i}}(h_{X}(1,\psi_{i}^{\sigma}))$.  For each $\sigma \in (\mathbb{Z}/\ell^{i}\mathbb{Z})^{\times}$, there exists a positive integer $c$ relatively prime with $\ell$ such that
$${\rm ord}_{\mathfrak{p}_{i}}(h_{X}(1,\psi_{i}^{\sigma})) = {\rm ord}_{\mathfrak{p}_{i}}(h_{X}(1,\psi_{i}^{c})).$$
Since it follows from (\ref{main_l}) that
$$h_{X}(1,\psi) = {\rm det}(D - A_{\psi}),$$
\cite[Corollary $5.3$]{McGown/Vallieres:2022a} implies that
\begin{equation*}
\tau(h_{X}(1,\psi_{i}^{c})) = f(\xi_{\ell^{i}}^{c}).
\end{equation*}
From now on, we let $f_{c}(T) = f(T^{c}) \in \mathbb{Z}[T;\mathbb{Z}_{\ell}]$ and $g_{c}(T) = g(T^{c})$.  Clearly $\mu(f_{c}(T)) = \mu$, $f_{c}(T) = p^{\mu}\cdot g_{c}(T)$ and furthermore for a fixed positive integer $i$ the reduction of $g(T)$ modulo $p$ has a primitive $\ell^{i}$-th root of unity in $k_{\infty}$ as a root if and only if the reduction of $g_{c}(T)$ modulo $p$ has a primitive $\ell^{i}$-th root of unity in $k_{\infty}$ as a root.  Therefore, 
\begin{equation*}
v_{p}(f(\xi_{\ell^{i}}^{c})) = v_{p}(f_{c}(\xi_{\ell^{i}})) = \mu,
\end{equation*}
when $i \ge n_{0}$.  We obtain
\begin{equation*}
\begin{aligned}
{\rm ord}_{p}(\kappa_{n}) &= C + \sum_{i = n_{0}}^{n}v_{p}(\tau(h_{X}(1,\Psi_{i}))) \\ 
&= C + \sum_{i = n_{0}}^{n} \varphi(\ell^{i}) \cdot \mu,
\end{aligned}
\end{equation*}
and the result follows.
\end{proof}

Let us make a few remarks.  Firstly, if $\mu=0$ in Theorem \ref{main_1}, then ${\rm ord}_{p}(\kappa_{n})$ becomes constant for $n$ sufficiently large.  Secondly, we point out that Washington mentioned on page $192$ of \cite{Washington:1975} that ``it  is  reasonable  to  conjecture  that  $e_{n} = \mu p^{n} + \nu$ for  large  $n$,  for  suitable  constants $\mu$ and $\nu$''.  (In \cite{Washington:1975}, the roles of $\ell$ and $p$ are reversed, and $e_{n}$ denotes the $\ell$-adic valuation of the class number $h_{n}$ in a $\mathbb{Z}_{p}$-extension of number field.)  It is perhaps interesting to notice that this is exactly the behavior we obtained in Theorem \ref{main_1}.  At last, to complete the proof of Theorem~\ref{thmA}, let us briefly mention how every abelian $\ell$-tower as in (\ref{gen_tower}) is isomorphic (in a suitable sense) to one of the form (\ref{concrete_tower}) for some function $\alpha:S \longrightarrow \mathbb{Z}_{\ell}$.  Starting with an abelian $\ell$-tower as in (\ref{gen_tower}), one has group morphisms $\mu_{n}:H_{1}(X,\mathbb{Z}) \longrightarrow \Gamma_{n} = {\rm Gal}(X_{n}/X)$ for all $n \ge 1$, since $H_{1}(X,\mathbb{Z})$ classifies connected abelian covers of $X$.  Furthermore, these maps are compatible in the sense that $\pi_{n+1} \circ \mu_{n+1} = \mu_{n}$ for all $n \ge 1$, where $\pi_{n+1}:\Gamma_{n+1} \longrightarrow \Gamma_{n}$ is the natural projection map.  Therefore, we get a group morphism
$$\mu:H_{1}(X,\mathbb{Z}) \longrightarrow  \lim_{\substack{\longleftarrow \\ n \ge1 }}\Gamma_{n}\simeq \mathbb{Z}_{\ell}.$$
Let now $T$ be a spanning tree of $X$, and choose a section $S$ as explained at the beginning of \S\ref{Sec:main_1}.  For each $s \in S$, consider $c_{s}$, the unique geodesic path (meaning no backtracks) going from the end point of $s$ to the starting point of $s$ within $T$ followed by $s$, and let $\langle c_{s} \rangle$ be the corresponding cycle in $H_{1}(X,\mathbb{Z})$.  Define $\alpha:S \longrightarrow \mathbb{Z}_{\ell}$ via $s \mapsto \alpha(s) = \mu(\langle c_{s}\rangle)$.  Then, one can show that there exist isomorphisms of multigraphs 
$$\phi_{n}:X(\mathbb{Z}/\ell^{n}\mathbb{Z},S,\alpha_{n}) \stackrel{\simeq}{\longrightarrow} X_{n}$$
such that all the squares and the triangle in the diagram
\begin{equation*}
\begin{tikzcd}
         & \arrow[ld] X(\mathbb{Z}/\ell\mathbb{Z},S,\alpha_{1})  \arrow[dd, "\phi_{1}"]  & \arrow[l]X(\mathbb{Z}/\ell^{2}\mathbb{Z},S,\alpha_{2})   \arrow[dd, "\phi_{2}"] &\arrow[l] \ldots  &\arrow[l] \arrow[dd, "\phi_{n}"]  X(\mathbb{Z}/\ell^{n}\mathbb{Z},S,\alpha_{n}) & \arrow[l]  \ldots\\  
         X &  & & & & \\
          & \arrow[lu]X_{1}  & \arrow[l]X_{2}   & \arrow[l] \ldots  & \arrow[l] X_{n}  \arrow[l] & \arrow[l] \ldots
\end{tikzcd}
\end{equation*}
commute.  We leave the details to the reader.

\begin{corollary} \label{prime_showing}
With the same notation as in Theorem \ref{main_1}, if the reduction modulo $p$ of the generalized polynomial $g(T)$ has no roots in $\mu_{\ell^{\infty}}(k_{\infty}) \smallsetminus \{1\}$, then
$${\rm ord}_{p}(\kappa_{n}) = \mu \cdot \ell^{n} - \mu + {\rm ord}_{p}(\kappa_{X}),$$
for all $n \ge 1$. 
\end{corollary}
\begin{proof}
Indeed, in Theorem \ref{main_1}, we have $n_{0}=1$, therefore
$${\rm ord}_{p}(\kappa_{n}) = {\rm ord}_{p}(\kappa_{X}) + \sum_{i=1}^{n}\varphi(\ell^{i}) \cdot \mu = \mu \cdot \ell^{n} - \mu + {\rm ord}_{p}(\kappa_{X}), $$
as we wanted to show.
\end{proof}

\begin{corollary}\label{nec_suf_cond}
With the same notation as in Theorem \ref{main_1}, assume furthermore that $\mu = 0$.  Then $p \nmid \kappa_{n}$ for all $n \ge 1$ if and only if $p \nmid \kappa_{X}$ and the reduction modulo $p$ of the generalized polynomial $g(T)$ has no roots in $\mu_{\ell^{\infty}}(k_{\infty}) \smallsetminus \{1\}$.
\end{corollary}
\begin{proof}
This follows from the formula
$${\rm ord}_{p}(\kappa_{n}) = {\rm ord}_{p}(\kappa_{X}) + \sum_{i=1}^{n_{0}} v_{p}(\tau(h_{X}(1,\Psi_{i}))),$$
noticing that the numbers $h_{X}(1,\psi)$ are algebraic integers so that $v_{p}(\tau(h_{X}(1,\Psi_{i}))) \ge 0$.  Since there exists $i \in \{1,\ldots,n_{0} \}$ such that
$$v_{p}(\tau(h_{X}(1,\Psi_{i}))) > 0$$
if and only if the reduction modulo $p$ of the generalized polynomial $g(T)$ has a root in $\mu_{\ell^{\infty}}(k_{\infty}) \smallsetminus \{1\}$, the result follows.
\end{proof}

%\section{The number of primes dividing $\kappa_n$ in  a special case}
\section{The number of primes dividing $\kappa_n$ in  a special case}\label{Sec:main_2}
In this section, we study the particular case where $\alpha(S) \subseteq \mathbb{Z}$ in more details.  The main difference between this situation and the more general one treated so far is that one can work with Laurent polynomials (or just polynomials after clearing the denominators) rather than generalized polynomials.  First note that the generalized polynomial $f(T)$ of Theorem \ref{main_1} satisfies
$$f(T) = f\left(\frac{1}{T} \right), $$
since the matrix $M(1/T)$ is the transpose of $M(T)$.  Therefore, under the assumption $\alpha(S) \subseteq \mathbb{Z}$, one has
$$f(T) = a_{-b}T^{-b} + a_{-(b-1)}T^{-(b-1)} +  \ldots + a_{0} + \ldots + a_{b-1}T^{b-1} + a_{b}T^{b}, $$
for some $b \in \mathbb{N}$ and some $a_{i} \in \mathbb{Z}$ satisfying $a_{i} = a_{-i}$ for all $i = 0,1,\ldots,b$.  We assume without lost of generality that $a_{b} \neq 0$.  It follows that
$$T^{b} \cdot f(T) \in \mathbb{Z}[T] $$
is a palindromic polynomial of even degree $2b$.  Set
$$U(T) = T^{b} \cdot f(T) \in \mathbb{Z}[T].$$
Since $a_{b} \neq 0$, one has $U(0)=a_{b} \neq 0$.
\begin{proposition} \label{when_stabilize}
With the same notation as in Theorem \ref{main_1}, assume that we are in the situation where $\mu = 0$ and that $\alpha(S) \subseteq \mathbb{Z}$.  Let $f_{i}$ be the inertia index of $p$ in $\mathbb{Q}(\zeta_{\ell^{i}})$ and let $r_{i}$ be the number of prime ideals lying above $p$ in $\mathbb{Q}(\zeta_{\ell^{i}})$.  We let $\bar{U}(T) \in \mathbb{F}_{p}[T]$ denote the reduction of $U(T)$ modulo $p$.  If $n_{1}$ is the smallest positive integer such that
$$f_{n_{1}} > {\rm deg}(\bar{U}(T)), $$
then there exists an integer $\nu$ such that 
$${\rm ord}_{p}(\kappa_{n}) = \nu,$$
when $n \ge n_{1}$.
\end{proposition}
\begin{proof}
First, note that the reduction $\bar{\Phi}_{\ell^{i}}(T)$ modulo $p$ of the $\ell^{i}$-th cyclotomic polynomial factors into a product of $r_{i}$ irreducible polynomials of degree $f_{i}$.  If $\xi$ is a primitive $\ell^{i}$-th root of unity in $\mathbb{Q}_{p}(\xi_{\ell^{\infty}})$ such that
$$U(\xi) \equiv 0 \pmod{\mathfrak{m}_{\infty}}, $$
then one would have $V(T) \, | \, \bar{U}(T)$, where $V(T)$ denotes the minimal polynomial of $\xi + \mathfrak{m}_{\infty}$ over $\mathbb{F}_{p}$.  Thus, we would have
$$f_{i} = {\rm deg}(V(T)) \le {\rm deg}(\bar{U}(T)). $$
The sequence $(f_{i})_{i=1}^{\infty}$ being increasing, we get that if $i \ge n_{1}$, then
$$U(\xi) \not \equiv 0 \pmod{\mathfrak{m}_{\infty}} $$
for any primitive $\ell^{i}$-th root of unity whenever $i \ge n_{1}$.  Just as in the proof of Theorem \ref{main_1}, it follows that
$${\rm ord}_{\mathfrak{p}_{i}}(h_{X}(1,\Psi_{i})) = 0, $$
when $i \ge n_{1}$ and the result follows.
\end{proof}
Recall that every prime $p$ finitely decompose in $\mathbb{Q}(\zeta_{\ell^{\infty}})$.  From some point on, there will be a constant number $r$ of prime ideals lying above $p$ in $\mathbb{Q}(\zeta_{\ell^{i}})$.  Since $r_{i} \le r$ for all $i \in \mathbb{N}$, we obtain that ${\rm ord}_{p}(\kappa_{n})$ stabilizes once
\begin{equation}\label{level_stabilize}
n > {\rm log}_{\ell}\left(\frac{r \ell}{\ell-1} {\rm deg}(\bar{U}(T))\right),
\end{equation}
provided $\mu = 0$.  Here, $\log_\ell$ denotes the real logarithmic function in base $\ell$.
\begin{remark}
One could also say something about when the stabilization happens in the case where $\alpha(S) \not \subseteq \mathbb{Z}$ (and $\mu=0$) by extracting the relevant information from the proof of \cite[Theorem $2.2$]{Sinnott:1987}.  We leave this to the interested reader.
\end{remark}
One has that $1$ is a root of $U(T)$, since $M(1)$ is the Laplacian matrix of $X$ which is singular.  Let $m$ be the multiplicity of the root $1$.  Recall that if $V_{i}(T) \in \mathbb{Z}[T]$ are polynomials for $i=1,2,3$ and 
$$V_{3}(T) = V_{1}(T) \cdot V_{2}(T), $$
then if $V_{1}(T)$ and $V_{2}(T)$ are palindromic, so is $V_{3}(T)$.  Furthermore, if $V_{3}(T)$ and either $V_{1}(T)$ or $V_{2}(T)$ is palindromic, then so is the other.  Therefore, if we let
$$U_{1}(T) = \frac{U(T)}{(T-1)^{m}}, $$
then $U_{1}(T)$ is also a palindromic polynomial.  

Let $\alpha$ be a root of $U_{1}(T)$, if any.  It is an algebraic number and consider the number field $K = \mathbb{Q}(\alpha)$.  Recall that if $\beta$ and $\gamma$ are algebraic integers in $K$, then one writes $(\beta,\gamma)=1$ if $(\beta) + (\gamma) = O_{K}$, where $O_{K}$ is the ring of integers of $K$, and one says that $\beta$ and $\gamma$ are relatively prime.  Since $O_{K}$ is a Dedekind domain, the condition $(\beta,\gamma)=1$ is equivalent to saying that there is no prime ideal of $O_{K}$ that appears in both the factorization of $(\beta)$ and $(\gamma)$.  Since $K$ is the fraction field of $O_{K}$, there exist $\beta, \gamma \in O_{K}$ such that $\alpha = \beta/\gamma$.  Consider now the ideal $(\beta,\gamma)$.  If it is principal, say $(\beta,\gamma) = (\delta)$ for some $\delta \in O_{K}$, then we have $\beta = \beta_{1} \delta$ and $\gamma = \gamma_{1} \delta$ for some $\beta_{1},\gamma_{1} \in O_{K}$.  One has also $\alpha = \beta_{1}/\gamma_{1}$, but now $(\beta_{1},\gamma_{1}) = 1$.  If $(\beta,\gamma)$ is not principal, then it might not be possible to write $\alpha$ as the quotient of two relatively prime algebraic integers in $O_{K}$.  On the other hand, there exists a finite extension $L/K$ for which the ideal $(\beta,\gamma)$ becomes principal.  (The Hilbert class field would do for instance.)  Therefore, in $L$, it is possible to write $\alpha$ as a quotient of two algebraic integers in $O_{L}$ that are relatively prime to one another.

We are now ready to prove Theorem~\ref{thmB} stated in the introduction:
\begin{theorem} \label{number_of_primes}
Let $\omega$ be the usual prime omega function giving the number of distinct prime factors of a natural number.  With the same notation as above,
$$\lim_{n \to \infty}\omega(\kappa_{n}) = \infty, $$
if and only if $U(T)$ has a root in $\overline{\mathbb{Q}}$ that is not a root of unity.
\end{theorem}
\begin{proof}
Assume first that $U(T)$ has a root that is not a root of unity, say $\alpha$.  For any abelian cover $Y/X$ of multigraphs, we have $\kappa_{X} \, | \, \kappa_{Y}$ by \cite[Corollary $4.15$]{Baker/Norine:2009}.  (See also \cite[Corollary $4.10$]{Hammer:2020}.)  Now, note that there are only finitely many rational primes $p$ such that $\mu_{p} \neq 0$, since a prime $p$ for which $\mu_{p} \neq 0$ necessarily divides the content of $U(T)$.  There are also finitely many rational primes $p$ dividing $\kappa_{X}$.  In order to show the claim it suffices, by Corollary \ref{nec_suf_cond}, to show that there are infinitely many rational primes $p$ for which $\mu_{p} = 0$ and for which the reduction of $U(T)$ modulo $p$ has a root in $\mu_{\ell^{\infty}}(\overline{\mathbb{F}}_{p}) \smallsetminus \{1\} \subseteq \mu_{\ell^{\infty}}(k_{\infty}) \smallsetminus \{ 1\}$.  Consider the number field $K = \mathbb{Q}(\alpha)$, and let $L$ be a finite extension of $K$ in which $\alpha$ can be written as $\beta/\gamma$ for two algebraic integers in $O_{L}$ satisfying $(\beta,\gamma)=1$.  We are interested in the prime ideals $\mathfrak{p}$ of $L$ for which $(\beta  \gamma,\mathfrak{p}) = 1$ and 
\begin{equation} \label{cong}
\beta^{\ell^{i}} \equiv \gamma^{\ell^{i}} \pmod{\mathfrak{p}}, 
\end{equation}
for some $i \ge 1$.  Indeed, recall that for such a prime ideal $\mathfrak{p}$, we have a well-defined group morphism $\rho:L_{(\mathfrak{p})}^{\times} \longrightarrow \kappa(\mathfrak{p})^{\times}$, where $L_{(\mathfrak{p})}^{\times}$ denotes the subgroup of $L^{\times}$ consisting of elements relatively prime with $\mathfrak{p}$ and $\kappa(\mathfrak{p})$ denotes the residue field of $L$ at $\mathfrak{p}$.  One has $\rho(\alpha) = \bar{\beta} \cdot \bar{\gamma}^{-1},$
where the bar denotes reduction modulo $\mathfrak{p}$.  The congruence (\ref{cong}) implies that $\rho(\alpha) \in \mu_{\ell^{\infty}}(\overline{\mathbb{F}}_{p})\smallsetminus \{1\}$ ($p$ is the rational prime lying below $\mathfrak{p}$), and
$$ \bar{g}(\rho(\alpha)) = \bar{f}(\rho(\alpha)) =  \rho(\alpha)^{-b} \cdot\bar{U}(\rho(\alpha)) = \bar{0}. $$
Let $\delta_{i} = \beta^{\ell^{i}}-\gamma^{\ell^{i}}$ and consider
$${\rm Supp}(\delta_{i}) = \{ \mathfrak{p} \, : \, \mathfrak{p} \, | \, \delta_{i}\}.$$
Let also
$$m_{i} = |{\rm Supp}(\delta_{i}) \smallsetminus {\rm Supp}(\delta_{0})|.$$
Note that ${\rm Supp}(\delta_{i}) \subseteq {\rm Supp}(\delta_{j})$ whenever $i \le j$, and that the sequence $(m_{i})_{i=1}^{\infty}$ is increasing.  Since by assumption, $\alpha$ is not a root of unity and $(\beta,\gamma) = 1$, a theorem of Schinzel (see \cite[Theorem $1$]{Schinzel:1974}) generalizing Zsigmondy's theorem to number fields implies that $m_{i} \to \infty$ as $i \to \infty$.  Since there is at most $[L:\QQ]$ primes lying above a fixed $p$, we do get that $\omega(\kappa_{n}) \to \infty$ as $n \to \infty$.

Conversely, if all the roots, say $\alpha_{1},\ldots, \alpha_{2b}$ of $U(T)$ are roots of unity, let $L$ be a cyclotomic number field containing all of the roots.  The numbers
$$\alpha_{j}^{\ell^{i}}-1 $$
are divisible by at most finitely many prime ideals in $L$ as $j$ runs over $\{1,\ldots,2b \}$ and $i \in \mathbb{N}$.  The result then follows.
\end{proof}

In the next section, we give some explicit examples where Theorem~\ref{number_of_primes} applies. In particular, we shall see in Examples \ref{example_finite_1} and \ref{example_finite}  that $\omega(\kappa_n)$ are bounded as $n\rightarrow \infty$, whereas $\displaystyle\lim_{n\rightarrow \infty}\omega(\kappa_n)=\infty$ in Examples \ref{ex:unbounded1}, \ref{ex:unbounded2} and \ref{ex:unbounded3}.

%\section{Examples}
\section{Examples}\label{Sec:Ex}
The simplest situation is when the base multigraph is a bouquet $B_{t}$, in which case $S = \{s_{1},\ldots,s_{t}\}$ denotes the loops of $B_{t}$, and/or when $\alpha(S) \subseteq \mathbb{Z}$.  The numbers of spanning trees in this section have been calculated with the software \cite{SAGE}.
\begin{Ex}
Let $\ell = 5$.  \label{example_finite_1}We take $X = B_{3}$ the bouquet with three loops, and we let $\alpha$ be defined by $\alpha(s_{j}) = 1$ for all $j=1,2,3$.  We have an abelian $5$-tower over $B_{3}$
$$
B_{3} \longleftarrow X(\mathbb{Z}/5\mathbb{Z},S,\alpha_{1}) \longleftarrow X(\mathbb{Z}/5^{2}\mathbb{Z},S,\alpha_{2}) \longleftarrow \ldots \longleftarrow X(\mathbb{Z}/5^{n}\mathbb{Z},S,\alpha_{n}) \longleftarrow \ldots
$$
The power series $Q(T)$ of \cite[Theorem $6.1$]{McGown/Vallieres:2022a} can be calculated to be
$$Q(T) = -3T^{2} -3T^{3} -3T^{4} -3T^{5}- \ldots \in \mathbb{Z}\llbracket T\rrbracket \subseteq \mathbb{Z}_{5}\llbracket T \rrbracket, $$
so that $\mu_{5} = 0$ and $\lambda_{5} = 1$.  One has
$${\rm ord}_{5}(\kappa_{n}) = n ,$$
for all $n \ge 1$.  (This also follows from \cite[Corollary $5.7$]{Vallieres:2021}.)  We have now
$$f(T) = 6 -3(T+T^{-1}) = -3T^{-1} + 6 -3T \text{ and } U(T) = T\cdot f(T)  = - 3(T-1)^{2}. $$
Therefore, $\mu_{3} = 1$ and $\mu_{p}=0$ for all $p \neq 3$.  Furthermore, Corollary \ref{prime_showing} shows that for $n \ge 1$ we have
$${\rm ord}_{3}(\kappa_{n}) = 5^{n} - 1, $$
and Corollary \ref{nec_suf_cond} shows that $p \nmid \kappa_{n}$ for all primes $p \neq 3,5$.  Using SageMath, we calculate
$$\kappa_{0} = 1, \kappa_{1} = 3^{4} \cdot 5, \kappa_{2} = 3^{24} \cdot 5^{2}, \kappa_{3} = 3^{124} \cdot 5^{3} , \kappa_{4} = 3^{624} \cdot 5^{4}, \ldots$$

\end{Ex}
\begin{Ex}\label{ex:unbounded1} Let $\ell=3$.  We take $X=B_{4}$ the bouquet with four loops, and we let $\alpha$ be defined by $\alpha(s_{j}) = 1$ for $j=1,2$ and $\alpha(s_{j}) = 2$ for $j=3,4$.  We have an abelian $3$-tower over $B_{4}$
$$
B_{4} \longleftarrow X(\mathbb{Z}/3\mathbb{Z},S,\alpha_{1}) \longleftarrow X(\mathbb{Z}/3^{2}\mathbb{Z},S,\alpha_{2}) \longleftarrow \ldots \longleftarrow X(\mathbb{Z}/3^{n}\mathbb{Z},S,\alpha_{n}) \longleftarrow \ldots
$$
The power series $Q(T)$ of \cite[Theorem $6.1$]{McGown/Vallieres:2022a} can be calculated to be
$$Q(T) = -10T^{2} - 10T^{3} - 12T^{4} -14T^{5} \ldots \in \mathbb{Z}\llbracket T\rrbracket \subseteq \mathbb{Z}_{3}\llbracket T \rrbracket, $$
so that $\mu_{3} = 0$ and $\lambda_{3} = 1$.  One has
$${\rm ord}_{3}(\kappa_{n}) = n,$$
for all $n \ge 1$.  (This also follows from \cite[Corollary $5.7$]{Vallieres:2021}.)  We have now
\begin{equation*}
\begin{aligned}
f(T) &= 8 - 2(T + T^{-1}) - 2(T^{2}+T^{-2}) \\
&= -2T^{-2} - 2T^{-1} + 8 -2T - 2T^{2},
\end{aligned}
\end{equation*}
and
$$U(T) = T^{2} \cdot f(T) = -2(T-1)^{2}(1+3T+T^{2}).$$
Therefore, $\mu_{2}=1$ and $\mu_{p} =0$ for all $p \neq 2$.    Furthermore, Corollary \ref{prime_showing} shows that ${\rm ord}_{2}(\kappa_{n})$ will be unbounded whereas ${\rm ord}_{p}(\kappa_{n})$ will remain bounded when $p \neq 2,3$.  Using SageMath, we calculate
\begin{equation*}
\begin{aligned}
\kappa_{0} &= 1\\
\kappa_{1} &= 2^{4} \cdot 3\\
\kappa_{2} &= 2^{10} \cdot 3^{2} \cdot 17^{2}\\
\kappa_{3} &= 2^{28} \cdot 3^{3} \cdot 17^{2} \cdot 53^{2} \cdot 109^{2}\\
\kappa_{4} &= 2^{82} \cdot 3^{4} \cdot 17^{2} \cdot 53^{2} \cdot 109^{2} \cdot 2269^{2} \cdot 4373^{2} \cdot 19441^{2}\\
&\hspace{4cm} \vdots 
\end{aligned}
\end{equation*}
Since $\Phi_{3}(T) \equiv 1+3T+T^{2} \pmod{2}$, we have $n_{0}=2$ in Theorem \ref{main_1} when $p=2$.  Thus
$${\rm ord}_{2}(\kappa_{n}) = 3^{n} + 1,$$
an equality which is true for all $n \ge 1$.  

Eventually, there are $3$ primes lying above $17$ in $\mathbb{Q}(\zeta_{\ell^{i}})$ as $i \to \infty$.  Using Proposition \ref{when_stabilize}, one has
$${\rm ord}_{17}(\kappa_{n})=2 $$
for all $n \ge 2$.  The other primes showing up can be studied similarly.  For instance eventually there are $9$ primes lying above $53$ in $\mathbb{Q}(\zeta_{\ell^{i}})$ as $i \to \infty$ and one has ${\rm ord}_{53}(\kappa_{n})=2$ for all $n \ge 3$.  Also, eventually there are $18$ primes lying above $109$ in $\mathbb{Q}(\zeta_{\ell^{i}})$ as $i \to \infty$ and ${\rm ord}_{109}(\kappa_{n})=2$ for all $n \ge 3$.

The roots of $1 + 3T + T^{2}$ are given by
$$\frac{-3 \pm \sqrt{5}}{2}, $$
which are algebraic integers that are not roots of unity.  Therefore, Theorem \ref{number_of_primes} applies, and we have $\omega(\kappa_{n}) \to \infty$ as $n \to \infty$.

\end{Ex}

\begin{Ex}
Let $\ell=5$.  \label{example_finite}We take the multigraph $X$ consisting of two vertices $\{ v_{1},v_{2}\}$ and three parallel edges.  We take a section $S$ by directing the edge $s_{1}$ from $v_{1}$ to $v_{2}$ and the edges $s_{2},s_{3}$ from $v_{2}$ to $v_{1}$.  We let $\alpha$ be defined by $\alpha(s_{1}) = 1$ and $\alpha(s_{2}) = \alpha(s_{3})=2$.  We have an abelian $5$-tower of connected multigraphs over $X$:
$$
X \longleftarrow X(\mathbb{Z}/5\mathbb{Z},S,\alpha_{1}) \longleftarrow X(\mathbb{Z}/5^{2}\mathbb{Z},S,\alpha_{2}) \longleftarrow \ldots \longleftarrow X(\mathbb{Z}/5^{n}\mathbb{Z},S,\alpha_{n}) \longleftarrow \ldots
$$
The power series $Q(T)$ of \cite[Theorem $6.1$]{McGown/Vallieres:2022a} can be calculated to be
$$Q(T) = -18T^{2} - 18T^{3}-30T^{4}-42T^{5}- \ldots \in \mathbb{Z}\llbracket T\rrbracket \subseteq \mathbb{Z}_{5}\llbracket T \rrbracket, $$
so that $\mu_{5} = 0$ and $\lambda_{5} = 1$.  One has
$${\rm ord}_{5}(\kappa_{n}) = n,$$
for all $n \ge 1$.  We calculate now
\begin{equation*}
M(T) = 
\begin{pmatrix}
3 & -(T + 2T^{-2}) \\
-(T^{-1}+2T^{2}) & 3
\end{pmatrix},
\end{equation*}
and
$$f(T) = -2T^{-3} + 4 - 2T^{3}.$$
Thus
$$U(T) =  T^{3} \cdot f(T) = - 2(T-1)^{2}\Phi_{3}(T)^{2}.$$
Therefore, $\mu_{p} = 0$ for all $p \neq 2$ and it follows that ${\rm ord}_{p}(\kappa_{n})$ will eventually stabilize for all $p \neq 2,5$.  Furthermore, Theorem \ref{number_of_primes} shows that $\omega(\kappa_{n})$ is bounded as $n \to \infty$ and Corollary \ref{nec_suf_cond} shows that at most the rational primes $2,3$ and $5$ will eventually divide $\kappa_{n}$.  Corollary \ref{prime_showing} shows further that we have
$${\rm ord}_{2}(\kappa_{n}) = 5^{n}-1.$$
Using SageMath, we calculate
\begin{equation*}
\begin{aligned}
\kappa_{0} &= 3\\
\kappa_{1} &= 2^{4} \cdot 3 \cdot 5\\
\kappa_{2} &= 2^{24} \cdot 3 \cdot 5^{2} \\
\kappa_{3} &= 2^{124} \cdot 3 \cdot  5^{3}\\
\kappa_{4} &= 2^{624} \cdot 3 \cdot 5^{4} \\
&\hspace{1cm} \vdots 
\end{aligned}
\end{equation*}
 
\end{Ex} 

\begin{Ex} \label{ex:unbounded2}Let $\ell=3$.  We take $X=B_{4}$ the bouquet with four loops, and we let $\alpha$ be defined by $\alpha(s_{1}) = 1$, and $\alpha(s_{i}) = 2$ for $i=2,3,4$.  We have an abelian $3$-tower over $B_{4}$
$$
B_{4} \longleftarrow X(\mathbb{Z}/3\mathbb{Z},S,\alpha_{1}) \longleftarrow X(\mathbb{Z}/3^{2}\mathbb{Z},S,\alpha_{2}) \longleftarrow \ldots \longleftarrow X(\mathbb{Z}/3^{n}\mathbb{Z},S,\alpha_{n}) \longleftarrow \ldots
$$
The power series $Q(T)$ of \cite[Theorem $6.1$]{McGown/Vallieres:2022a} can be calculated to be
$$Q(T) = -13T^{2} -13T^{3} - 16T^{4} -19T^{5} \ldots  \in \mathbb{Z}\llbracket T\rrbracket \subseteq \mathbb{Z}_{3}\llbracket T \rrbracket, $$
so that $\mu_{3} = 0$ and $\lambda_{3} = 1$.  One has
$${\rm ord}_{3}(\kappa_{n}) = n,$$
for all $n \ge 1$.  (This also follows from \cite[Corollary $5.7$]{Vallieres:2021}.)  We calculate now
\begin{equation*}
\begin{aligned}
f(T) &= 8 - (T + T^{-1}) - 3(T^{2} + T^{-2}) \\
&= -3T^{-2} - T^{-1} + 8 - T - 3T^{2},
\end{aligned}
\end{equation*}
and
$$U(T) = T^{2} \cdot f(T) = -(T-1)^{2}(3 + 7T + 3T^{2}).$$
Therefore, $\mu_{p} = 0$ for all $p \neq 3$ and it follows that ${\rm ord}_{p}(\kappa_{n})$ will eventually stabilize for all $p \neq 3$.  Using SageMath, we calculate
\begin{equation*}
\begin{aligned}
\kappa_{0} &= 1\\
\kappa_{1} &= 2^{4} \cdot 3\\
\kappa_{2} &= 2^{4} \cdot 3^{2} \cdot 127^{2}\\
\kappa_{3} &= 2^{4} \cdot 3^{3} \cdot 127^{2} \cdot 3295783^{2}\\
\kappa_{4} &= 2^{4} \cdot 3^{4} \cdot 127^{2} \cdot 1621^{2} \cdot 3295783^{2} \cdot 22480434859526947^{2}\\
&\hspace{4cm} \vdots 
\end{aligned}
\end{equation*}
A root of $3 + 7T + 3T^{2}$ is given by
$$\alpha = \frac{-7 + \sqrt{13}}{6} \in \mathbb{Q}(\sqrt{13}),$$
which can be written as $\alpha = \beta/\gamma$, where
$$\beta = -\left( \frac{5+\sqrt{13}}{2}\right) \text{ and } \gamma = 4 + \sqrt{13}.$$ 
Note that $\beta$ and $\gamma$ are algebraic integers in $\mathbb{Q}(\sqrt{13})$ that satisfy $(\beta,\gamma)=1$, and $\alpha$ is not a root of unity.  Therefore, Theorem \ref{number_of_primes} applies, and we have $\omega(\kappa_{n}) \to \infty$ as $n \to \infty$

\end{Ex}

\begin{Ex}\label{ex:unbounded3} Let $\ell=2$.  We take the multigraph $X$ consisting of two vertices $\{ v_{1},v_{2}\}$ and four parallel edges.  We take a section $S$ by directing all of the edges from $v_{1}$ to $v_{2}$, and we let $\alpha$ be defined by $\alpha(s_{i}) = i$ for $i=1,2,3,4$.  We have an abelian $2$-tower of connected multigraphs over $X$:
$$
X \longleftarrow X(\mathbb{Z}/2\mathbb{Z},S,\alpha_{1}) \longleftarrow X(\mathbb{Z}/2^{2}\mathbb{Z},S,\alpha_{2}) \longleftarrow \ldots \longleftarrow X(\mathbb{Z}/2^{n}\mathbb{Z},S,\alpha_{n}) \longleftarrow \ldots
$$
The power series $Q(T)$ of \cite[Theorem $6.1$]{McGown/Vallieres:2022a} can be calculated to be
$$Q(T) = -20T^{2} - 20T^{3} - 28T^{4} -36T^{5} - 45T^{6} - \ldots \in \mathbb{Z}\llbracket T\rrbracket \subseteq \mathbb{Z}_{2}\llbracket T \rrbracket, $$
so that $\mu_{2} = 0$ and $\lambda_{2} = 5$.  One has
$${\rm ord}_{2}(\kappa_{n}) = 5n + 2,$$
for all $n \ge 2$.  We calculate now
\begin{equation*}
M(T) = 
\begin{pmatrix}
4 & -(T + T^{2} + T^{3} + T^{4}) \\
-(T^{-1}+T^{-2} +T^{-3} +T^{-4}) & 4
\end{pmatrix},
\end{equation*}
and
$$f(T) = -T^{-3} - 2T^{-2} - 3T^{-1} + 12 - 3T - 2T^{2} - T^{3}.$$
Thus
$$U(T) =  T^{3} \cdot f(T) = - (T-1)^{2}(1+4T+10T^{2} + 4T^{3} + T^{4}).$$
Therefore, $\mu_{p} = 0$ for all $p \neq 2$ and it follows that ${\rm ord}_{p}(\kappa_{n})$ will eventually stabilize for all $p \ge 3$.  Using SageMath, we calculate
\begin{equation*}
\begin{aligned}
\kappa_{0} &= 2^{2}\\
\kappa_{1} &= 2^{5}\\
\kappa_{2} &= 2^{12} \\
\kappa_{3} &= 2^{17} \cdot 17^{2}\\
\kappa_{4} &= 2^{22} \cdot 17^{2} \cdot 1217^{2} \\
\kappa_{5} &= 2^{27} \cdot 17^{2} \cdot 257^{2} \cdot 1217^{2} \cdot 23041^{2}\\
\kappa_{6} &= 2^{32} \cdot 17^{2} \cdot 257^{2} \cdot 1217^{2} \cdot 23041^{2} \cdot 158209^{2} \cdot 886538753^{2}\\
%\kappa_{7} &= 2^{37} \cdot 17^{2} \cdot 257^{2} \cdot 769^{2} \cdot 1217^{2} \cdot 19457^{2} \cdot 23041^{2} \cdot 158209^{2} \cdot 886538753^{2} \cdot 5259150421874111701249^{2}\\
&\hspace{4cm} \vdots 
\end{aligned}
\end{equation*}
The roots of the polynomial $1+4T+10T^{2} + 4T^{3} + T^{4}$ are not roots of unity and it follows from Theorem \ref{number_of_primes} that $\omega(\kappa_{n}) \to \infty$ as $n \to \infty$.\end{Ex}
\begin{Ex} Let $\ell=2$.  We take $X = B_{2}$ the bouquet with two loops, and we let $\alpha$ be defined by $\alpha(s_{1}) = \sqrt{17} = 1.0010\ldots \in \mathbb{Z}_{2}$, and $\alpha(s_{2}) = 5$.  We have an abelian $2$-tower of connected multigraphs over $B_{2}$:
$$
B_{2} \longleftarrow X(\mathbb{Z}/2\mathbb{Z},S,\alpha_{1}) \longleftarrow X(\mathbb{Z}/2^{2}\mathbb{Z},S,\alpha_{2}) \longleftarrow \ldots \longleftarrow X(\mathbb{Z}/2^{n}\mathbb{Z},S,\alpha_{n}) \longleftarrow \ldots
$$
The power series $Q(T) \in \mathbb{Z}_{2}\llbracket T \rrbracket$ of \cite[Theorem $6.1$]{McGown/Vallieres:2022a} can be calculated to be
$$(0.1101\ldots) T^{2} + (0.1101\ldots) T^{3} + (0.0011\ldots) T^{4} + (0.1011\ldots)T^{5} + (1.1101\ldots)T^{6}+ \ldots, $$
so that $\mu_{2} = 0$ and $\lambda_{2} = 5$.  One has
$${\rm ord}_{2}(\kappa_{n}) = 5n -3,$$
for all $n \ge 3$.  We have now
$$f(T) = 4 - (T^{\sqrt{17}} + T^{-\sqrt{17}}) - (T^{5} + T^{-5}) \in \mathbb{Z}[T;\mathbb{Z}_{2}].$$
Therefore, $\mu_{p} = 0$ for all $p \neq 2$ and it follows that ${\rm ord}_{p}(\kappa_{n})$ will eventually stabilize for all $p \ge 3$.  Using SageMath, we calculate
\begin{equation*}
\begin{aligned}
\kappa_{0} &= 1\\
\kappa_{1} &= 2^{2}\\
\kappa_{2} &= 2^{5} \\
\kappa_{3} &= 2^{12}\\
\kappa_{4} &= 2^{17} \cdot 17^{2} \\
\kappa_{5} &= 2^{22} \cdot 17^{2} \cdot 1217^{2}\\
\kappa_{6} &= 2^{27} \cdot 17^{2} \cdot 257^{2} \cdot 1217^{2} \cdot 23041^{2}\\
\kappa_{7} &= 2^{32} \cdot 17^{2} \cdot 257^{4} \cdot 1217^{2} \cdot 23041^{2} \cdot 1518337^{2} \cdot 27744257^{2}\\
&\hspace{4cm} \vdots 
\end{aligned}
\end{equation*}
Since $\alpha(S)$ is not contained in $\ZZ$, Theorem~\ref{number_of_primes} does not apply. However, our numerical evidence seems to suggest that $\omega(\kappa_n)$ is unbounded as $n\rightarrow\infty$. 
\end{Ex}

\subsection*{Statement on conflict of interest}
On behalf of all authors, the corresponding author states that there is no conflict of interest. 

\subsection*{Data availability statement}
Data sharing not applicable to this article as no datasets were generated or analysed during the current study.
\bibliographystyle{plain}
\bibliography{main}

\end{document}